\documentclass[reqno,a4paper,draft]{amsart}

\usepackage{enumitem}
\setenumerate{label=\textnormal{(\arabic*)}}

\usepackage{amsmath,amssymb,dsfont,verbatim,bm,array,mathtools}
\usepackage[latin1]{inputenc}
\usepackage{booktabs}
\usepackage[raggedright]{titlesec}
\usepackage{mathtools}

\titleformat{\chapter}[display]
{\normalfont\huge\bfseries}{\chaptertitlename\\thechapter}{20pt}{\Huge}
\titleformat{\section}
{\normalfont\Large\bfseries\center}{\thesection}{1em}{}
\titleformat{\subsection}
{\normalfont\large\bfseries}{\thesubsection}{1em}{}
\titleformat{\subsubsection}[runin]
{\normalfont\normalsize\bfseries}{\thesubsubsection}{1em}{}
\titleformat{\paragraph}[runin]
{\normalfont\normalsize\bfseries}{\theparagraph}{1em}{}
\titleformat{\subparagraph}[runin]
{\normalfont\normalsize\bfseries}{\thesubparagraph}{1em}{}
\titlespacing*{\chapter} {0pt}{50pt}{40pt}
\titlespacing*{\section} {0pt}{3.5ex plus 1ex minus .2ex}{2.3ex plus .2ex}
\titlespacing*{\subsection} {0pt}{3.25ex plus 1ex minus .2ex}{1.5ex plus .2ex}
\titlespacing*{\subsubsection}{0pt}{3.25ex plus 1ex minus .2ex}{1.5ex plus .2ex}
\titlespacing*{\paragraph} {0pt}{3.25ex plus 1ex minus .2ex}{1em}
\titlespacing*{\subparagraph} {\parindent}{3.25ex plus 1ex minus .2ex}{1em}

\input xypic
\xyoption{all}


\newtheorem{theorem}{Theorem}[section]
\newtheorem{lemma}[theorem]{Lemma}
\newtheorem{proposition}[theorem]{Proposition}

\theoremstyle{definition}

\theoremstyle{remark}
\newtheorem{remark}[theorem]{Remark}

\DeclareMathOperator{\Char}{Char}

\DeclareMathOperator{\Jac}{Jac}

\DeclareMathOperator{\Frac}{Frac}
\DeclareMathOperator{\Spec}{Spec}
\DeclareMathOperator{\Img}{Img}
\DeclareMathOperator{\Ker}{Ker}

\begin{document}
\title{Special cases of the Jacobian conjecture}
\author{Vered Moskowicz}
\address{Department of Mathematics, Bar-Ilan University, Ramat-Gan 52900, Israel.}
\email{vered.moskowicz@gmail.com}
\thanks{The author was partially supported by an Israel-US BSF grant \#2010/149}

\begin{abstract}
The famous Jacobian conjecture asks if a morphism $f:K[x,y]\to K[x,y]$
that satisfies
$\Jac(f(x),f(y))\in K^*$ is invertible
($K$ is a characteristic zero field).
We show that if one of the following three equivalent conditions is satisfied, then $f$ is invertible: \begin{itemize}
\item $K[f(x),f(y)][x+y]$ is normal.
\item $K[x,y]$ is flat over $K[f(x),f(y)][x+y]$.
\item $K[f(x),f(y)][x+y]$ is separable over $K[f(x),f(y)]$.
\end{itemize}
\end{abstract}

\maketitle

\section{Introduction}
Throughout this paper, $K$ is an algebraically closed field of characteristic zero
(sometimes there is no need to assume that $K$ is algebraically closed),
the field of fractions of an integral domain $R$ is denoted by $\Frac(R)$,
$f:K[x,y]\to K[x,y]$ is a morphism that satisfies
$\Jac(f(x),f(y)) \in K^*$,
and 
$P:= f(x)$, $Q:= f(y)$.
$\Jac(P,Q) \in K^*$ implies that $P$ and $Q$ are algebraically independent over $K$
(see, for example, \cite[Proposition 6A.4]{rowen}), hence 
$K[P,Q]$ is isomorphic to the $K$-algebra of polynomials in two commuting indeterminates.
$P,Q$ and $x$ are algebraically dependent over $K$, hence $x$ is algebraic over 
$K[P,Q] \subset K(P,Q)$. Similarly, $y$ is algebraic over $K(P,Q)$. Therefore,
$K(P,Q) \subseteq K(P,Q)(x,y)= K(x,y)$ is a finite field extension.
Since $\Char(K)=0$, $K(P,Q) \subseteq K(x,y)$ is a separable field extension.
Apply the primitive element theorem to the finite separable field extension
$K(P,Q) \subseteq K(x,y)$, and get that there exists $w \in K(x,y)$ such that 
$K(x,y)= K(P,Q)(w)$; such $w$ is called a primitive element for the extension.    
A standard proof of the primitive element theorem which does not use Galois theory (see, for example, 
\cite[Theorem 1]{web}) shows that
$K(x,y)=K(P,Q)(x+\lambda y)$, $\lambda \in K(P,Q)$, for all but finitely many choices of 
$\lambda \in K(P,Q)$. So
$K(x,y)=K(P,Q)(x+\lambda y)$ for infinitely many $K \ni \lambda$'s; we call such $\lambda$'s
``good".

\begin{proposition}\label{x+y}
$K(x,y)= K(P,Q)(x+y)$.
\end{proposition}

\begin{proof}
This is just \cite[Exercise]{web} which claims that $\lambda=1$ is one of the infinitely many 
``good" $\lambda$'s.
\end{proof}

For convenience, we shall always work with $x+y$ as a primitive element, 
though we could have taken any other ``good" $\lambda \in K$.
(Without \cite[Exercise]{web}, one takes any ``good" $\lambda$, denote it $\lambda_0$,
and works with $K[P,Q][x+\lambda_0y]$ instead of $K[P,Q][x+y]$).


\section{Preliminaries}
Recall the following important theorems:

\textbf{Bass's theorem} \cite[Proposition 1.1]{bass2}: Assume that $K[x_1,x_2] \subseteq B$ 
is an affine integral domain over $K$ which is an unramified extension of $K[x_1,x_2]$.
Assume also that $B=K[x_1,x_2][b]$ for some $b \in B$. If $B^*=K^*$ then $B=K[x_1,x_2]$.

\textbf{Formanek's theorem} \cite[Theorem 1]{formanek}, see also \cite[page 13, Exercise 9]{essen},
which is true for any characteristic zero field $K$, not necessarily algebraically closed: 
If $F_1, \ldots, F_n \in K[x_1, \ldots, x_n]$ 
satisfy $\Jac(F_1, \ldots, F_n) \in K^*$ and there exists 
$G \in K[x_1, \ldots, x_n]$ such that 
$K[F_1, \ldots, F_n, G]= K[x_1, \ldots, x_n]$,
then $K[F_1, \ldots, F_n]= K[x_1, \ldots, x_n]$.
In particular, take $n=2$ in Formanek's theorem and get:
Let $K$ be a characteristic zero field.
If $F_1, F_2 \in K[x,y]$ satisfy $\Jac(F_1, F_2) \in K^*$ 
and there exists $G \in K[x,y]$ such that $K[F_1, F_2, G]= K[x,y]$,
then $K[F_1, F_2]= K[x,y]$.

(A special case of) \textbf{Adjamagbo's transfer theorem} \cite[Theorem 1.7]{adja.essen}: Given commutative rings $A \subseteq B \subseteq C$ such that: $A$ is normal and Noetherian,
$B$ is isomorphic to $A[T]/hA[T]$, where $A[T]$ is the $A$-algebra of polynomials generated by one indeterminate $T$ and $h \in A[T]-A$, $C$ an affine $B$-algebra, $C$ is separable over $A$,
$C^*=A^*$ and $\Spec(C)$ is connected.
Then the following conditions are equivalent:\begin{itemize}
\item $B$ is normal.
\item $C$ is flat over $B$.
\item $B$ is separable over $A$.
\item $B$ is \'{e}tale (=unramified and flat) over $A$.
\end{itemize}

\begin{lemma}\label{Keller}
Each of the following special cases implies that $f$ is invertible: \begin{itemize}
\item [(1)] $K[P,Q][x+y]=K[x,y]$.
\item [(2)] $K[P,Q][x+y]=K[P,Q]$.
\end{itemize}
\end{lemma}

\begin{proof}
\begin{itemize}
\item [(1)] From Formanek's theorem we get $K[x,y]=K[P,Q]$.
\item [(2)] 
$$
K(x,y)= K(P,Q)(x+y)= \Frac(K[P,Q][x+y])= \Frac(K[P,Q])= K(P,Q).
$$
The first equality follows from Proposition \ref{x+y}, the others are obvious.
Then Keller's theorem \cite[Corollary 1.1.35]{essen} (see also \cite[Theorem 2.1]{bass}) 
says that $K[x,y]=K[P,Q]$.
\end{itemize}
\end{proof}

\section{Main Theorem}
Recall the following well-known results: \begin{itemize}
\item [(1)] $K[P,Q]$ and $K[x,y]$ are normal.
\item [(2)] $K[x,y]$ is flat over $K[P,Q]$ (``flatness of Keller maps"),
see \cite[Theorem 38]{wang} or \cite[Corollary 1.1.34]{essen}. 
(There exist commutative rings $A \subseteq B \subseteq C$ such that
$C$ is faithfully flat over $A$, $B$ is faithfully flat over $A$,
but $C$ is not flat over $B$; for more details, see \cite[D.1.4 and Exercise D.2.4]{essen} and 
\cite[page 49, Example]{mat}. In our case 
$K[P,Q] \subseteq K[P,Q][x+y] \subseteq K[x,y]$, we do not 
even know if $K[x,y]$ is faithfully flat over $K[P,Q]$ or if $K[P,Q][x+y]$ is faithfully flat over 
$K[P,Q]$). 
\item [(3)] $K[x,y]$ is separable over $K[P,Q]$, see \cite[Theorem 7, Theorem 38]{wang},
\cite[Proposition 1.10]{wright} and \cite[pages 295-296]{bass}.
Notice that separability of $K[x,y]$ over $K[P,Q]$ implies separability of $K[x,y]$ over 
$K[P,Q][x+y]$, see \cite[page 92 (13)]{adja.essen}.
\end{itemize}

\begin{theorem}\label{main thm}
If one of the following equivalent conditions is satisfied, then $f$ is invertible:
\begin{itemize}
\item [(1)] $K[P,Q][x+y]$ is normal.
\item [(2)] $K[x,y]$ is flat over $K[P,Q][x+y]$.
\item [(3)] $K[P,Q][x+y]$ is separable over $K[P,Q]$.
\end{itemize}
\end{theorem}

\begin{proof}
The commutative rings $K[P,Q] \subseteq K[P,Q][x+y] \subseteq K[x,y]$ satisfy all the conditions in Adjamagbo's theorem:
$K[P,Q]$ is isomorphic to the $K$-algebra of polynomials in two commuting indeterminates, hence it is normal (a UFD is normal) and Noetherian. $K[P,Q][x+y]$ is isomorphic to 
$K[P,Q][T]/hK[P,Q][T]$, where $h \in K[P,Q][T]-K[P,Q]$ is the minimal polynomial of $x+y$
over $K[P,Q]$. $K[x,y]=K[P,Q][x+y][y]$ is an affine 
$K[P,Q][x+y]$-algebra. $K[x,y]$ is separable over $K[P,Q]$, as was already mentioned.
$\Spec(K[x,y])$ is connected, since the prime spectrum of any integral domain is.

{}From Adjamagbo's theorem, the three conditions are indeed equivalent, and are also equivalent to 
$K[P,Q][x+y]$ being \'{e}tale over $K[P,Q]$; in particular 
$K[P,Q][x+y]$ is unramified over $K[P,Q]$. 
Now apply Bass's theorem to $K[P,Q] \subseteq K[P,Q][x+y]$ and get that
$K[P,Q][x+y]=K[P,Q]$. By $(2)$ of Lemma \ref{Keller} $f$ is invertible.
\end{proof}
\section{Special cases of the main theorem}
A special case of Theorem \ref{main thm} when $K[x,y]$ is faithfully flat over 
$K[P,Q][x+y]$, has an easier proof: Recall that if $A$ and $B$ are integral domains, 
$A \subseteq B$, $\Frac(A)= \Frac(B)$, and $B$ is faithfully flat over $A$, then $A=B$ 
(see \cite[Exercise 7.2]{mat}). Apply this to $A=K[P,Q][x+y]$ and $B=K[x,y]$,
and get $K[P,Q][x+y]=K[x,y]$.
By $(1)$ of Lemma \ref{Keller} $f$ is invertible.
Another special case of Theorem \ref{main thm} is the following:

\begin{theorem}\label{thm regular}
If $K[P,Q][x+y]$ is regular, then $f$ is invertible. 
\end{theorem}

Recall that a commutative ring is regular if it is Noetherian and the localization at every 
prime ideal is a regular local ring. 
For the definition of a regular local ring see, for example, \cite[pages 104-105]{mat}
and \cite[page 123]{ati}.

\begin{proof}
Recall that a regular ring is normal \cite[Theorem 19.4]{mat}, 
hence $K[P,Q][x+y]$ is normal.
Then Theorem \ref{main thm} implies that $f$ is invertible.
\end{proof}

\begin{remark}\label{remark}
$K[P,Q][x+y]$ is Noetherian: 
$K[P,Q][x+y] \cong K[P,Q,S]/ gK[P,Q,S]$,
where $S$ is transcendental over $K[P,Q]$ and 
$g=g(S) \in K[P,Q][S]$ is minimal such that
$g(x+y)=0$ ($x+y$ is algebraic over $K[P,Q]$).
Hence $K[P,Q][x+y]$ is Noetherian as a homomorphic image of the Noetherian polynomial ring
$K[P,Q,S]$ in three indeterminates $P,Q,S$.
\end{remark}

In view of Theorem \ref{thm regular}, one wishes to show that 
$K[P,Q][x+y]$ is regular. 

If a commutative Noetherian ring $R$ has finite Krull dimension and finite global dimension,
then $R$ is regular and these dimensions are equal, \cite[Theorem 11.2]{may}.

Therefore,
\begin{theorem}
If the global dimension of $K[P,Q][x+y]$ is finite, 
then $f$ is invertible.
\end{theorem} 

\begin{proof}
It is enough to show that $K[P,Q][x+y]$ has finite Krull dimension 
(then \cite[Theorem 11.2]{may} and Theorem \ref{thm regular} imply that $f$ is invertible).

{}From \cite[Theorem 5.6]{mat} and Proposition \ref{x+y}, the Krull dimension of 
$K[P,Q][x+y]$ is $2$. 

(Another way: It is known that the Krull dimension is at most the global dimension. 
Hence if the global dimension is finite so is the Krull dimension).
\end{proof}

We suspect that from 
$K[P,Q] \subseteq K[P,Q][x+y] \subseteq K[x,y]$ and
$K[P,Q][x+y] \cong K[P,Q,S]/ gK[P,Q,S]$,
one must get that the global dimension of $K[P,Q][x+y]$ is $2$.
\cite[Theorem 6.1]{global} tells us that $K[P,Q][x+y]$-modules of \textit{finite} projective dimension 
have projective dimension $\leq 2$
(since otherwise, the polynomial ring in three indeterminates $K[P,Q,S]$ would have a module of projective dimension $\geq 4$, a contradiction).
 
However, it may happen (though we hope that it may not happen) that there exist 
$K[P,Q][x+y]$-modules of infinite projective dimension.  
In other words, we wish to be able to show that every finitely generated
$K[P,Q][x+y]$-module has finite projective dimension (hence necessarily projective dimension 
$\leq 2$. The option of an infinite increasing chain of finite dimensions is impossible, 
as we have just remarked that a $K[P,Q][x+y]$-module of finite projective dimension must have 
projective dimension $\leq 2$).

We hope that it is possible to show that every finitely generated $K[P,Q][x+y]$-module
has finite projective dimension, namely that every finitely generated $K[P,Q][x+y]$-module
has a finite projective resolution.

A special case of every finitely generated $K[P,Q][x+y]$-module having a 
finite projective dimension is that of every finitely generated $K[P,Q][x+y]$-module 
having a finite free resolution (FFR);
if so, then $K[P,Q][x+y]$ is even a UFD (any UFD is normal) 
by \cite[Theorem 184]{kap} or \cite[Theorem 20.4]{mat}
(=if a Noetherian integral domain $R$ has the property that every finitely generated 
$R$-module has an FFR, then $R$ is a UFD).

We hope that it is possible to show that every finitely generated $K[P,Q][x+y]$-module has an FFR, 
or at least a finite projective resolution; thus far we only managed to show that every finitely generated 
$K[P,Q][x+y]$-module has a finite \textit{complex} of finitely generated and free 
$K[P,Q][x+y]$-modules- we will elaborate on this in Lemma \ref{lemma}.
We hope that this finite complex somehow yields a finite projective resolution (or a finite free resolution). 

Finally, we wonder if the following ``plausible theorem" is true: If a Noetherian integral domain $R$ has the property that every finitely generated $R$-module has a finite complex of finitely generated 
and free $R$-modules, then $R$ is normal.
We do not know if the ``plausible theorem" is true, but if it is true then $f$ is invertible:

\begin{theorem}\label{plausible}
If the ``plausible theorem" is true, then $f$ is invertible.
\end{theorem}

In order to prove Theorem \ref{plausible}, we need the following lemma:
\begin{lemma}\label{lemma}
Let $R$ be a commutative ring and $I$ an ideal of $R$. 
Assume $R$ has the property that every finitely generated $R$-module has a finite complex 
of finitely generated and free $R/I$-modules.
Then the quotient ring $R/I$ has the property that every finitely generated $R/I$-module
has a finite complex of finitely generated and free $R/I$-modules.
\end{lemma}

\begin{proof}
Let $M=M_{R/I}$ be any finitely generated $R/I$-module. We must show that $M_{R/I}$ has a
finite complex of finitely generated and free $R/I$-modules. 
Assume $m_1,\ldots,m_t$ generate $M_{R/I}$ as an $R/I$-module.
$M_{R/I}$ becomes an $R$-module by defining for all $r \in R$ and $m \in M_{R/I}$:
$rm:=(r+I)m$. Denote this new $R$-module by $M_R$. 
Notice that $I$ is in the annihilator of $M_R$. 
Clearly, $M_R$ is finitely generated over $R$ by the same $m_1,\ldots,m_t$.
By our assumption, $M_R$ has a finite complex of finitely generated and free $R$-modules:
$0 \rightarrow F_n  \stackrel{\alpha_n}{\rightarrow} \ldots \stackrel{\alpha_2}{\rightarrow}
F_1 \stackrel{\alpha_1}{\rightarrow} F_0 \stackrel{\alpha_0}{\rightarrow} M_R \rightarrow 0$,
where each $F_i$ is finitely generated and free $R$-module
and $\Img(\alpha_{j+1}) \subseteq \Ker(\alpha_{j})$.

For each $0 \leq j \leq n$ define $G_j:=F_j/IF_j$. 
Each $G_j$ is an $R/I$-module by defining for all 
$r+I \in R/I$ and $g_j=f_j+IF_j \in G_j$:
$(r+I)(f_j+IF_j):= rf_j+IF_j$. 

This is well-defined: If $r+I=r'+I$ and $f_j+IF_j=f'_j+IF_j$, 
write $r'=r+a$ and $f'_j=f_j+f$
where $a \in I$ and $f \in IF_j$.
Then 
$(r'+I)(f'_j+IF_j)= r'f'_j+IF_j=(r+a)(f_j+f)+IF_j= rf_j+(rf+af_j+af)+IF_j= rf_j+IF_j$,
since $rf+af_j+af \in IF_j$.

$F_j$ is a finitely generated $R$-module, hence $G_j$ is a finitley generated
$R/I$-module, with generators images of the generators of $F_j$ via 
$F_j \rightarrow F_j/IF_j$.

$F_j$ is $R$-free, hence $G_j$ is $R/I$-free. 
Indeed, if $F_j$ is a free $R$-module with basis 
$f_{j,1},\ldots,f_{j,s_j}$, then $G_j$ is a free $R/I$-module with basis
$f_{j,1}+IF_j,\ldots,f_{j,s_j}+IF_j$: 
If
$(r_1+I)(f_{j,1}+IF_j)+\ldots+(r_{s_j}+I)(f_{j,s_j}+IF_j)=0+IF_j$,
we must show that $r_1+I=0+I,\ldots,r_{s_j}+I=0+I$.
Clearly, $r_1f_{j,1}+\ldots+r_{s_j}f_{j,s_j} \in IF_j$.

So we can write
$r_1f_{j,1}+\ldots+r_{s_j}f_{j,s_j}= a_1u_1+\ldots+a_lu_l$, with 
$a_1,\ldots,a_l \in I$ and $u_1,\ldots,u_l \in F_j$.

For each $1 \leq i \leq l$, write $u_i=e_{i,1}f_{j,1}+\ldots+e_{i,s_j}f_{j,s_j}$,
where $e_{i,1},\ldots,e_{i,s_j} \in R$.

Then
$r_1f_{j,1}+\ldots+r_{s_j}f_{j,s_j}= a_1(e_{1,1}f_{j,1}+\ldots+e_{1,s_j}f_{j,s_j})+
\ldots+ a_l(e_{l,1}f_{j,1}+\ldots+e_{l,s_j}f_{j,s_j})
= (a_1e_{1,1}+\ldots+a_le_{l,1})f_{j,1}+\ldots+(a_1e_{1,s_j}+\ldots+a_le_{l,s_j})f_{j,s_j}$.

For all $1 \leq k \leq s_j$, write
$b_k:=a_1e_{1,k}+\ldots+a_le_{l,k} \in I$.

Therefore, 
$r_1f_{j,1}+\ldots+r_{s_j}f_{j,s_j}= b_1f_{j,1}+\ldots+b_{s_j}f_{j,s_j}$.
So,
$(r_1-b_1)f_{j,1}+\ldots+(r_{s_j}-b_{s_j})f_{j,s_j}= 0$.

But $f_{j,1},\ldots,f_{j,s_j}$ are free over $R$, hence 
$(r_1-b_1)=0,\ldots,(r_{s_j}-b_{s_j})=0$, namely
$r_1=b_1 \in I,\ldots,r_{s_j}=b_{s_j} \in I$,
so we have
$r_1+I=0+I,\ldots,r_{s_j}+I=0+I$.

Next, the given 
$F_j \stackrel{\alpha_j}{\rightarrow} F_{j-1}$
are yielding
$F_j/IF_j \stackrel{\beta_j}{\rightarrow} F_{j-1}/IF_{j-1}$,

where
$\beta_j(f_j+IF_j):=\alpha_j(f_j)+IF_{j-1}$.
Notice that $\beta_j$ is well-defined:

$\beta_j(f'_j+IF_j)=\alpha_j(f'j)+IF_{j-1}=\alpha_j(f_j+f)+IF_{j-1}
=\alpha_j(f_j)+\alpha_j(f)+IF_{j-1}= \alpha_j(f_j) +IF_{j-1}= \beta_j(f_j+IF_j)$,
where $f \in IF_j$. 

($f= c_1f_{j,1}+\ldots+c_{s_j}f_{j,s_j}$ with $c_1,\ldots,c_{s_j} \in I$,
so 
$\alpha_j(f)=\alpha_j(c_1f_{j,1}+\ldots+c_{s_j}f_{j,s_j})
=c_1\alpha_j(f{j,1})+\ldots+c_{s_j}\alpha_j(f_{j,s_j}) \in IF_{j-1}$.

It remains to show that 
$0 \rightarrow G_n  \stackrel{\beta_n}{\rightarrow} \ldots \stackrel{\beta_2}{\rightarrow}
G_1 \stackrel{\beta_1}{\rightarrow} G_0 \stackrel{\beta_0}{\rightarrow} M_{R/I} \rightarrow 0$
is a complex, namely: 
$\Img(\beta_{j+1}) \subseteq \Ker(\beta_j)$.

Take
$\beta_{j+1}(w+IF_{j+1}) \in \Img(\beta_{j+1})=\beta_{j+1}(F_{j+1}/IF_{j+1})$,
where $w \in F_{j+1}$.

We know that
$\alpha_{j+1}(w) \in \Img(\alpha_{j+1}) \subseteq \Ker(\alpha_j)$,
hence $\alpha_j(\alpha_{j+1}(w))=0$.

Then
$\beta_j(\beta_{j+1}(w+IF_{j+1}))=\beta_j(\alpha_{j+1}(w)+IF_j)
= \alpha_j(\alpha_{j+1}(w))+IF_{j-1}=0+IF_{j-1}$,
so  
$\Img(\beta_{j+1}) \subseteq \Ker(\beta_j)$.
\end{proof}

Now we move to prove Theorem \ref{plausible}.
\begin{proof}
Write $R=K[P,Q,S]$ and $I=gK[P,Q,S]$ as in Remark \ref{remark}.
It is well-known (see, for example, \cite[Theorem 182, Theorem 183]{kap}) that the polynomial ring 
$K[P,Q,S]$ has the property that every finitely generated $K[P,Q,S]$-module has an FFR. 
In particular, every finitely generated $K[P,Q,S]$-module has a finite complex of
finitely generated and free $K[P,Q,S]$-modules.
$R/I=K[P,Q][x+y]$ is a Noetherian integral domain which, by Lemma \ref{lemma}, has the 
property that every finitely generated $K[P,Q][x+y]$-module has a finite complex of 
finitely generated and free $R$-modules.
Hence if the ``plausible theorem" is true, then $K[P,Q][x+y]$ is normal.
Then Theorem \ref{main thm} says that $f$ is invertible.
\end{proof}

\begin{remark}
Notice that in Lemma \ref{lemma}, $I$ is any ideal of $R$, while in our case 
$I=gK[P,Q,S]$ is a prime ideal of $R=K[P,Q,S]$. Maybe this fact can help in showing that
every finitely generated $K[P,Q][x+y]$-module has an FFR (and not just a finite complex of 
finitely generated and free $K[P,Q][x+y]$-modules).

We wish to prove a new version of Lemma \ref{lemma}, namely: 
Let $R$ be a commutative ring and $I$ a prime ideal of $R$. 
Assume $R$ has the property that every finitely generated $R$-module has an FFR.
Then the quotient ring $R/I$ has the property that every finitely generated $R/I$-module
has an FFR.

However, we were not able to prove this new version, since exactness in the $F_j$'s
seems not to imply exactness in the $G_j$'s.

If this new version is true, then the ``plausible theorem" is not needed,
and $K[P,Q][x+y]$ is a UFD by the already mentioned 
\cite[Theorem 184]{kap} or \cite[Theorem 20.4]{mat}.

\end{remark}
\section{Acknowledgements}
I would like to thank my PhD advisors Prof. Louis Rowen and Prof. Uzi Vishne for enlightening conversations about commutative rings. I would also like to thank Dr. Monique and Mordecai Katz for their Bar-Ilan university's stipends for PhD students and for learning at Bar-Ilan's Midrasha.

\bibliographystyle{plain}

\end{document}